\documentclass[12pt]{amsart}

\usepackage{amsfonts,amsthm,amsmath,amssymb,amscd,mathrsfs}
\usepackage{graphics}
\usepackage{indentfirst}
\usepackage{cite}
\usepackage{latexsym}
\usepackage[dvips]{epsfig}

\newtheorem{theorem}{Theorem}[section]
\newtheorem{remark}{Remark}[section]

\newtheorem{lemma}[theorem]{Lemma}
\newtheorem{pro}{Proposition}[section]

\newcommand{\bt}{\begin{theorem}}
\newcommand{\bl}{\begin{lemma}}
\newcommand{\el}{\end{lemma}}
\newcommand{\et}{\end{theorem}}

\newcommand{\bn}{\begin{eqnarray}}
\newcommand{\en}{\end{eqnarray}}
\newcommand{\bnn}{\begin{eqnarray*}}
\newcommand{\enn}{\end{eqnarray*}}

\newcommand{\ba}{\begin{aligned}}
\newcommand{\ea}{\end{aligned}}
\newcommand{\be}{\begin{equation}}
\newcommand{\ee}{\end{equation}}

\newcommand{\mcD}{\mathcal{D}}
\makeatletter
\@addtoreset{equation}{section}

\begin{document}
\title[Global Axisymmetric Solutions for Navier-Stokes System]{Global strong solutions to the
inhomogeneous incompressible Navier-Stokes system in the exterior of  a cylinder}
\author{Zhengguang Guo}

\address{School of Mathematics and Statistics, Huaiyin Normal University, Huai'an 223300, China; School of mathematical Sciences, Institute of Natural Sciences, Shanghai Jiao Tong University, 800 Dongchuan Road, Shanghai, China}
\email{gzgmath@163.com}

\author{Yun Wang}
\address{School of Mathematical Sciences, Center for dynamical systems and differential equations, Soochow University, Suzhou, China}
\email{ywang3@suda.edu.cn}

\author{Chunjing Xie}
\address{School of mathematical Sciences, Institute of Natural Sciences,
Ministry of Education Key Laboratory of Scientific and Engineering Computing,
and SHL-MAC, Shanghai Jiao Tong University, 800 Dongchuan Road, Shanghai, China}
\email{cjxie@sjtu.edu.cn}

\begin{abstract}
In this paper, the global strong axisymmetric solutions for  the inhomogeneous incompressible Navier-Stokes system are established in the exterior of a cylinder subject to the Dirichlet boundary conditions. Moreover, the  vacuum is allowed in these solutions. One of the key ingredients of the analysis is to obtain the ${L^{2}(s,T;L^{\infty}(\Omega))}$ bound for the velocity field, where the axisymmetry of the solutions plays an important role.  \end{abstract}

\maketitle

\section{Introduction and main results\label{intro}}

The mixture of incompressible and non-reactant
flows, flows with complex structure
fluids containing a melted substance, etc (\cite{Lions}), can be described by the following inhomogeneous incompressible
Navier-Stokes system
\begin{equation}
\left\{
\begin{array}{l}
(\rho \mathbf{u})_{t}+\mathrm{div}(\rho \mathbf{u}\otimes \mathbf{u})-\mu
\Delta \mathbf{u}+\nabla P=0,\ \ \ \ \mbox{in}\ \Omega \times \lbrack 0,T), \\
\rho _{t}+\mathrm{div}(\rho \mathbf{u})=0\text{ \ \ \ }\mbox{in}\ \Omega
\times \lbrack 0,T) \\
\mathrm{div}~\mathbf{u}=0,\ \ \ \ \mbox{in}\ \Omega \times \lbrack 0,T),%
\end{array}
\right.  \label{NS}
\end{equation}
where $\rho ,$ $\mathbf{u,}$ $P$, and $\mu $ are the density, velocity field,
pressure, and viscosity coefficient of fluid, respectively. In this paper, the viscosity coefficient is assumed to be a constant.
Without loss of generality, one assumes $\mu =1$. Furthermore, in the domain $\Omega$ where the fluid occupies,   the system \eqref{NS} is usually supplemented
with the following initial conditions and no slip boundary conditions
\begin{equation}
(\rho ,\rho \mathbf{u})|_{t=0}=(\rho _{0},\rho _{0}\mathbf{u}_{0})\text{ in }%
\Omega ;\quad \mathbf{u}=0\text{ on }\partial \Omega \times (0,T).
\label{IBVP}
\end{equation}%

 Since Leray's pioneering work \cite%
{Leray} on the global existence of weak solutions to the homogeneous
incompressible Navier-Stokes system (corresponding to the case $\rho\equiv 1$), there have been many important progresses on the homogeneous
incompressible Navier-Stokes system. When the initial density is away from vacuum, there is a  counterpart theory of inhomogeneous Navier-Stokes system to Leray's results. The global existence of weak solutions and local existence of strong solutions for inhomogeneous Navier-Stokes system  were established in \cite{Antontsev,AKM, Solo}.  Furthermore, the strong solution exists globally in two dimensional case \cite{AKM}. Recently, there are many studies on the well-posedness for the inhomogeneous Navier-Stokes system in various critical spaces, see\cite{abidi2,danchin1,danchin2,paicu1} and references therein.

When the vacuum is allowed, the local and global existence of
weak solutions to system (\ref{NS}) was established in \cite{Kim87, Simon}. However, the uniqueness and smoothness of weak solutions to the inhomogeneous Navier-Stokes system,
even for the two dimensional case, are still  open problems. This is very different from the two dimensional homogeneous Navier-Stokes system (\cite{Lady}). A local strong solution under some compatibility conditions on the initial data was established in \cite{Choe}.
More precisely, given $(\rho _{0},\mathbf{u}_{0})$ satisfying%
\begin{equation}
0\leq \rho _{0}\in L^{\frac{3}{2}}(\Omega )\cap H^2(\Omega ),\text{\
}\mathbf{u}_{0}\in H_{0}^{1}(\Omega )\cap H^{2}(\Omega ),
\label{initial rho}
\end{equation}%
and the compatibility conditions
\begin{equation}
-\mu \Delta \mathbf{u}_{0}+\nabla P_{0}=\rho _{0}^{\frac{1}{2}}\mathbf{g},%
\text{ and\ }\mathrm{div}\text{ }\mathbf{u}_{0}=0 \text{ \ in }\Omega ,
\label{compatibility}
\end{equation}%
with some $(P_{0},\mathbf{g})$ belonging to $D^{1,2}(\Omega )\times
L^{2}(\Omega ),$ there exists a unique local strong  solution $(\rho ,
\mathbf{u})$ to the initial boundary value problem \eqref{NS}--\eqref{IBVP}. For the further studies on local well-posedness of strong solutions for inhomogeneous Navier-Stokes system,  see \cite{mmas,jmfm,huangjde,zhangjde, LiJDE} and references therein.  A natural question is
whether the general local strong solutions away from the vacuum can be prolonged globally in time. Suppose
the local strong solution blows up in finite time $T^{\ast },$ a Serrin type
blow-up criterion was established in \cite{Kim},
\begin{equation}
\int_{0}^{T^{\ast }}||\mathbf{u}(t)||_{L_{w}^{r}}^{s}dt=\infty ,\text{ for
any }(r,s)\text{ with }\frac{2}{s}+\frac{n}{r}=1,\text{ }n<r\leq \infty ,
\label{Serrin}
\end{equation}%
where $n$ is the dimension of space, and $L_{w}^{r}$ is the weak $L^{r}$
space. With the aid of this blow up criterion, for the initial data even with the vacuum, the global strong solutions for the inhomogeneous Navier-Stokes system in two dimensional case were established in  \cite{Gui, HW, HW2}.

Global existence of strong solutions for three dimensional Navier-Stokes system even in the homogeneous case is a long standing challenging problem.  However, it was proved in
\cite{La,UI} that for the axisymmetric solutions without swirls,  the global
Leray-Hopf weak solution for homogeneous Navier-Stokes system is regular for all time $t>0$. The proof in \cite{La,UI} was
based on important facts that the vorticity $\mathbf{\omega }=\nabla
\times \mathbf{u}$ satisfies the maximum principle and the global a priori estimate
\begin{equation} \label{maximum}
\left\Vert \frac{\mathbf{\omega }}{r}\right\Vert _{L^{2}(\mathbb{R}%
^{3})}\leq \left\Vert \frac{\mathbf{\omega }_{0}}{r}\right\Vert _{L^{2}(%
\mathbb{R}^{3})}
\end{equation}%
holds. However, when the swirl velocity is present, the global well-posedness for the axisymmetric Navier-Stokes system  becomes much more difficult.  There are many important progresses on this problem, see \cite{HL,CSTY1, CSTY2, KNSS, fangarma,preprint, CL,JX,KPZ,LZ} and references
therein.
On the other hand, the significant partial regularity results in \cite{CKN} (see also
\cite{GKT, Lin, Seregin, TX}) assert that the one-dimensional Hausdorff
measure of the set for singular points is zero. This implies that the
singularity of axisymmetric solutions can only happen at the axis. When the domain is  the exterior of a cylinder,    the global
existence of unique axisymmetric strong solution was proved in \cite{Lady} and \cite{AS} when the no slip and Navier boundary conditions were supplemented, respectively. The crucial points for the analysis in \cite{Lady} and \cite{AS} are an interpolation inequality and  the maximum principle \eqref{maximum}, respectively.

 The axisymmetric solutions for inhomogeneous Navier-Stokes system without swirls were studied in \cite{Abidi} and references therein.  For the inhomogeneous Navier-Stokes system in  an exterior domain,   the local existence of weak solutions was proved in \cite{Padula} when the initial density is positive almost everywhere. The main goal of this paper is to study the global existence of axisymmetric strong solutions for the inhomogeneous Navier-Stokes system in the exterior of a cylinder subject to the no slip boundary conditions,
  where \eqref{maximum} may not be true and it seems difficult to  apply the interpolation inequality used in \cite{Lady}.
 Without loss of generality, in this paper, one assumes that
    \[
    \Omega=\{(x_1, x_2, x_3)\in \mathbb{R}^3: r^2=x_1^2+x_2^2 >1, x_3 \in \mathbb{R}\}.
    \]
    The key idea in this paper is to get some bound for $\|\mathbf{u}\|_{L^2(s, T; L^\infty(\Omega))}$ from the energy inequality, which corresponds to a regularity criterion of Serrin type \cite{Serrin}.

Before stating the main results in the paper, the following notations are introduced.   For $1\leq q\leq \infty $, let $L^{q}(\Omega )$
denote the usual scalar-valued and vector-valued $L^{q}$-space over $\Omega $%
. Let
\begin{equation}
W^{m,q}(\Omega )=\{\mathbf{u}\in L^{q}(\Omega ):D^{\alpha }\mathbf{u}\in
L^{q}(\Omega ),\ |\alpha |\leq m, m\in \mathbb{N}\}.  \notag
\end{equation}%
When $q=2$, one abbreviates
$H^{m}(\Omega )=W^{m,2}(\Omega )$. Denote the closure of $C_0^\infty(\Omega)$  in $H^{1}(\Omega
) $ by $H_{0}^{1}(\Omega )$. Let
\begin{equation}
C_{0,\sigma }^{\infty }(\Omega )=\left\{ \mathbf{u}\in C_{0}^{\infty
}(\Omega ):\mathrm{div}~\mathbf{u}=0,\ \mbox{in}\ \Omega \right\} .  \notag
\end{equation}%
Denote the closure of $C_{0,\sigma }^{\infty }(\Omega )$ in $H^{1}(\Omega )$
by $H_{0,\sigma }^{1}(\Omega )$.

Our main result can be stated as follows.
\begin{theorem}
\label{main-result} Let $(\rho _{0},\mathbf{u}_{0})$ be axisymmetric initial data and satisfy compatibility condition \eqref%
{compatibility} and the following regularity conditions
\begin{equation}\label{regular condition}
\rho_0\geq 0,\ \  \  \rho _{0}- \bar{\rho} \in L^{\frac{3}{2}}(\Omega )\cap
H^{2}(\Omega ),\ \ \ \mathbf{u}_{0}\in H^2(\Omega)\cap H^1_0(\Omega),
\end{equation}
with $\bar{\rho} >0 $ is a constant.
Then for every $T>0$, there exists a unique axisymmetric strong solution $(\rho ,\mathbf{u})$ to the problem \eqref{NS}--\eqref{IBVP} with%
\begin{eqnarray*}
\rho - \bar{\rho}  \in C([0,T]; L^{\frac32}(\Omega) \cap H^{2}(\Omega )), \ \mathbf{u}\in  C([0,T];H^1_{0, \sigma} (\Omega ))\cap L^\infty(0, T; H^2(\Omega)),
\end{eqnarray*}%
and%
\begin{equation*}
\nabla \mathbf{u}_{t}\in L^{2}([0,T]; L^2(\Omega )),\text{ \ }(\rho _{t},\sqrt{%
\rho }\mathbf{u}_{t})\in L^{\infty }([0,T];L^{2}(\Omega )).
\end{equation*}
\end{theorem}

There are a few remarks in order.
\begin{remark}
Together with the analysis in \cite{HW}, one can also show that this
result holds for inhomogeneous MHD equations.
\end{remark}

\begin{remark}
Together with the method in \cite{LiJDE} for the proof of the local existence of solutions for the inhomogeneous Navier-Stokes system in bounded domains even when the initial data violate the compatibility conditions (\ref{compatibility}),
the compatibility conditions (\ref{compatibility}) should also be removed in Theorem
\ref{main-result}. The major aim of this paper is to highlight the a priori estimate to get global strong axisymmetric solutions for the inhomogeneous Navier-Stokes system so that we try to avoid including a more complicated local existence result in this paper.
\end{remark}

\begin{remark}
The method in this  paper can also be used to prove the global existence of axisymmetric strong solutions to Navier-Stokes system in the exterior of a cylinder subject to Navier boundary condition.
\end{remark}

\begin{remark}
The analysis in this paper should be also helpful for the study on the helically symmetric flows.
\end{remark}

The rest of this paper is organized as follows. Some elementary results on axisymmetric functions, the critical Sobolev inequalities, and the estimates regarding Stokes equations are collected in Section \ref{sec pre}, which are important for the analysis in the whole paper. The proof of Theorem \ref{main-result} is presented in Section \ref{main proof} after one assumes the local well-posedness of the problem. In Section \ref{local result},
the local existence and uniqueness of strong solutions are sketched.

\section{Preliminaries\label{sec pre}}

For $(x_1, x_2, x_3)\in \mathbb{R}^3$, introduce the cylindrical coordinate%
\begin{equation*}
r=\sqrt{(x_1)^{2}+(x_2)^{2}},\quad \theta =\arctan \frac{x_2}{x_1},\quad z=x_3,
\end{equation*}%
and denote ${\boldsymbol{e}}_{r},$ ${\boldsymbol{e}}_{\theta },$ ${%
\boldsymbol{e}}_{z}$ the standard basis vectors in the cylindrical
coordinate:%
\begin{equation}
\mathbf{e}_{r}(\theta )=\left(
\begin{array}{c}
\cos \theta \\
\sin \theta \\
0%
\end{array}%
\right) ,\ \ \mathbf{e}_{\theta }(\theta)=\left(
\begin{array}{c}
-\sin \theta \\
\cos \theta \\
0%
\end{array}%
\right),\ \ {\boldsymbol{e}}_{z}=\left(
\begin{array}{l}
0 \\
0 \\
1%
\end{array}%
\right) .  \notag
\end{equation}

A function $f$ or a vector-valued function $\mathbf{u=(}u^{r},u^{\theta
},u^{z}\mathbf{)}$ is said to be axisymmetric if $f,$ $u^{r},$ $u^{\theta
}$ and $u^{z}$ do not depend on $\theta $:%
\begin{equation}
\mathbf{u}(x_1,x_2,x_3)=u^{r}(r,z)\mathbf{e}_{r}+u^{\theta }(r,z)\mathbf{e}%
_{\theta }+u^{z}(r,z)\mathbf{e}_{z}.  \notag
\end{equation}

The following lemma shows that for axisymmetric initial data the local
strong solution to (\ref{NS}) is also axisymmetric.

\begin{lemma}
\label{axisymmetric} Assume that the initial data $(\rho _{0}, \mathbf{u}_0)$ is axisymmetric.  Then the local strong
solution $(\rho ,\mathbf{u})$ to \eqref{NS}--\eqref{IBVP} is also axisymmetric.
\end{lemma}

\begin{proof}
For every $\eta \in \lbrack 0,2\pi )$,
define the rotation matrix%
\begin{equation}
R(\eta )=(\mathbf{e}_{r}(\eta ),\mathbf{e}_{\theta }(\eta ),\mathbf{e}_{z}).
\notag
\end{equation}%
Let
\begin{eqnarray*}
\varrho (x_1,x_2,x_3,t) =\rho (R(x_1,x_2,x_3),t) \quad \text{and}\quad \mathbf{v}(x_1,x_2,x_3,t) =R^{t}\mathbf{u}(R(x_1,x_2,x_3),t)
\end{eqnarray*}%
where \textquotedblleft $R^{t}$\textquotedblright\ is the transpose of the  matrix $R$. Since the inhomogeneous Navier-Stokes system \eqref{NS} is rotation invariant and $%
\rho _{0}$ and $\mathbf{u}_{0}$ are axisymmetric, it is easy to check that $%
(\varrho ,\mathbf{v})$ is also a solution to \eqref{NS} with the initial
data $(\rho _{0},\mathbf{u}_{0})$. Due to the uniqueness of strong solutions
to \eqref{NS}--\eqref{IBVP}, one has $\varrho =\rho $ and $\mathbf{v}=\mathbf{u}$. Hence
the solution $(\rho ,\mathbf{u})$ is axisymmetric.
\end{proof}

The following critical Sobolev inequality of Logarithmic type plays an important role to obtain  the bound of $\Vert \mathbf{u}\Vert _{L^{2}(s,T;L^{\infty
}(\Omega ))}$.

\begin{lemma}
\label{Brezis-Inequality} Suppose $D$ is a domain in $\mathbb{R}%
^{2} $, for every function $f\in L^{2}(s,t; H_0^1(D)) \cap
L^{2}(s,t;W^{1,q}(D ))$ with some $q>2$, and $s<t$, it holds that
\begin{equation}
\Vert f\Vert _{L^{2}(s,t;L^{\infty }(D))}\leq C\Vert \nabla f\Vert
_{L^{2}(s,t; L^2(D))}\left[ \ln (e + \Vert f\Vert
_{L^{2}(s,t;W^{1,q}(D))} ) \right] ^{\frac{1}{2}} + C,  \label{B-Inequality}
\end{equation}%
where $C$ is independent of the function $s$, $t$, and the domain $D$.
\end{lemma}
\begin{proof}
The inequality (\ref{B-Inequality}) has been proved in \cite{HW} for $D=\mathbb{R}
^{2}$. If $D$ is a domain in $\mathbb{R}^{2},$ it can be proved by zero extension, so we omit the details here.
\end{proof}

The next lemma gives the uniform regularity estimates for solutions to the Stokes
equations with Dirichlet boundary condition.
\begin{lemma}
\label{lemma lam2} Let $\mcD $ be a  domain of $\mathbb{R}^{3},$
whose boundary  is uniformly of class $C^{3}.$ Assume $\mathbf{u} \in H_{0, \sigma}^1 (\mcD) $ is a weak solution
to the following Stokes equations
\be \label{Stokes}
\left\{
\ba & -\Delta \mathbf{u} + \nabla P = \mathbf{f}, \ \ \ \mbox{in} \ \mcD, \\
& {\rm div}~\mathbf{u}= 0, \ \ \ \ \ \mbox{in}\ \mcD, \\
& \mathbf{u}= 0,\ \ \ \ \mbox{on}\ \partial \mcD.
\ea
\right.
\ee
Then  for $f\in L^q(\mcD)$, $1< q < \infty$, it holds that
\be \label{uniform-Stokes-2}
\| \mathbf{u}\|_{W^{2, q} (\mcD)} \leq C \|\mathbf{f}\|_{L^q(\mcD)} + C \|\mathbf{u}\|_{W^{1, q}(\mcD)},
\ee
where the constant $C$ depends only on $q$ and  the $C^{3}$-regularity of $\partial
\mcD $ (not on the size of $\partial \mcD $ or $\mcD $).
\end{lemma}
\begin{proof}
The proof of \eqref{uniform-Stokes-2} for the particular case $q=2$ can be found in \cite[Lemma 2.2]{Heywood}. With the aid of  ``local" estimate up to the boundary for the Stokes problem (\ref{Stokes}) (\cite{Lady, GS}), the estimate \eqref{uniform-Stokes-2} for the general $q>1$ can be proved in the same spirit of \cite{Heywood}.
 For readers' convenience, we give a sketch of the proof for the case with general $q>1$.

In a neighbourhood of a given point $%
\xi \in \partial \mcD ,$ let the boundary $\partial \mcD $ be
represented by $y_{3}=F(y_{1},y_{2})$ in local Cartesian coordinates $%
(y_{1},y_{2},y_{3})$ chosen so that the positive $y_{3}$-axis coincides with
the inward normal to $\partial \mcD $ at $\xi .$ Suppose that $d$ is a
sufficiently small number determined by $C^{3}$-regularity of $\partial
\mcD $ near $\xi$. Denote $Q=\{x\in \mcD :|y_{1}|,|y_{2}|<d$ and $%
F(y_{1},y_{2})<y_{3}<F(y_{1},y_{2})+2d\}$ and $Q^{\prime }=\{x\in \mcD
:|y_{1}|,|y_{2}|<d/2$ and $F(y_{1},y_{2})<y_{3}<F(y_{1},y_{2})+d\}$.
It follows from \cite{Lady, GS} that 
\begin{equation}
\|D^{2}\mathbf{u}\|_{L^{q}(Q^{\prime })}\leq C_{\partial ,d} \|  \mathbf{%
u}\|_{W^{1, q}(Q)}+C_{\partial ,d}\|\mathbf{f} \|_{L^{q}(Q)}.  \label{esti boundary1}
\end{equation}
Moreover, let $G^{\prime }=\{x\in \mcD :\text{dist}(x,\partial \mcD
)>d/2\}$ be an open bounded subset of $\mcD ,$ with $\bar{G}^{\prime
}\subset \mcD .$ Choose a function $\zeta \in C_{0}^{2}(\mcD )$
satisfying that $\zeta \equiv 1$ in $G^{\prime }$ and $0\leq \zeta \leq 1$
elsewhere in $\mcD $, such that $|\nabla \zeta |$ and $|\Delta \zeta |$
are bounded by some constant $C_{\partial ,d}$ which depends on $d$ and the $%
C^{2}$-regularity of $\partial \mcD .$ Then one has
\be   \label{esti boundary2} \ba
\| D^{2}\mathbf{u} \|_{L^{q}(G^{\prime })} & \leq C_{d} \|\nabla \mathbf{u}
\|_{L^{q}(\mcD )}+ C_{d}\| \mathbf{f} \|_{L^{q}(\mcD )}+ C_d  \|(\Delta
\zeta )\mathbf{u} \|_{L^{q}(\mcD )} + C_d \|\mathbf{u}\|_{W^{1, q}(\mcD)}, \\
& \leq C_{\partial, d} \|\mathbf{u}\|_{W^{1, q}(\mcD)} + C_d \| \mathbf{f} \|_{L^q(\mcD)}.
\ea
\ee

 On the other hand, since the boundary is uniformly of class $C^{3}$, $d$ is sufficiently small, and the mean curvature of the surface near the given point is bounded, then the boundary strip $\mcD -G^{\prime }$ can be covered with a collection of
\textquotedblleft cubes" $Q_{i}^{\prime },$ of the type described in (\ref%
{esti boundary1}) in such a way that no point of $\mcD $ belongs to more
than ten of the associated larger \textquotedblleft cubes" $Q_{i}$. Thus
it follows from (\ref{esti boundary1}) that
\begin{equation}\label{esti sum}
\begin{aligned}
\|D^{2}\mathbf{u} \|_{L^{q}(\mcD \setminus G^{\prime })}\leq
&\sum\limits_{i}\|D^{2}\mathbf{u}\|_{L^{q}(Q_{i}^{\prime })}  \notag \\
\leq &\sum\limits_{i}\left( C_{\partial ,d} \| \mathbf{u}%
\|_{W^{1, q} (Q_{i})} + C_{\partial ,d}\|\mathbf{f}\|_{L^{q}(Q_{i})}\right)  \notag
\\
\leq & 10\left( C_{\partial ,d} \|\mathbf{u}\|_{W^{1, q} (\mcD
)}+C_{\partial ,d} \| \mathbf{f} \|_{L^{q}(\mcD )}\right).
\end{aligned}
\end{equation}
This,  together with (\ref{esti boundary2}), implies the desired estimate%
\begin{equation}\nonumber
\|D^{2}\mathbf{u}\|_{L^{q}(\mcD )}\leq C_{\partial }\left( \|
\mathbf{u}\|_{W^{1, q} (\mcD )}+ \|\mathbf{f} \|_{L^{q}(\mcD )}\right),
\label{esti proj}
\end{equation}
since $d$ is determined by the $C^3$-regularity of $\partial \mcD$.
Hence the proof of the lemma is completed.
\end{proof}

\section{A priori estimate and the proof of the main result\label{main proof}}

This section  devotes to the proof of Theorem \ref{main-result}.
Given
initial data $(\rho _{0},\mathbf{u}_{0})$  satisfying (\ref{regular
condition}) and the compatibility condition \eqref{compatibility}, Theorem \ref{local} asserts that there exists a unique
local strong solution $(\rho ,\mathbf{u})$. According to Lemma \ref{axisymmetric}, the solution is axisymmetric. Define the quantity $\Phi (T)$
as follows%
\begin{eqnarray*}
\Phi (T) = \sup_{0\leq t\leq T}\left( \|\rho - \bar{\rho}\|_{L^{\frac32}(\Omega)} +\| \rho - \bar{\rho} \|_{H^{2}(\Omega )} + \|
\mathbf{u}\|_{H^{2}(\Omega )}^{2}\right) + \| \sqrt{\rho }\mathbf{u}
_{t}\|_{L^{\infty }(0,T;L^2 (\Omega ))}^{2} .
\end{eqnarray*}
Suppose this local strong solution blows up at some $T^{\ast }<\infty ,$  the key issue is to  prove that in fact there exists a constant $\bar{M}<\infty $ depending
only on the initial data and $T^{\ast }$ such that%
\begin{equation}
\sup_{0\leq T <  T^{\ast }}\Phi (T)\leq \bar{M}.  \label{main esti}
\end{equation}%
This, together with  Theorem \ref{local},  implies that the local strong solution can be extended beyond $T^{\ast
},$ and thus gives a contradiction. Therefore, the local strong solution does
not blow up in finite time.

\textbf{Proof Theorem \ref{main-result}:} First, it is easy to see that for the strong solutions,   the system \eqref{NS} is equivalent to
\begin{equation}
\left\{
\begin{array}{l}
\rho \mathbf{u}_{t}+\left( \rho \mathbf{u}\cdot \nabla \right) \mathbf{u}%
-\Delta \mathbf{u}+\nabla P=0,\ \ \ \ \mbox{in}\ \Omega \times \lbrack 0,T),
\\
\rho _{t}+\mathbf{u}\cdot \nabla \rho =0\text{ \ \ \ }\mbox{in}\ \Omega
\times \lbrack 0,T), \\
\mathrm{div}~\mathbf{u}=0,\ \ \ \ \mbox{in}\ \Omega \times \lbrack 0,T).%
\end{array}%
\right.  \label{NSE}
\end{equation}%

The proof is divided into 5 steps.

{\it Step 1.\ $L^{\infty }$
bound for $\rho $.} The second equation in \eqref{NSE} is in fact a transport
equation, due to the divergence free property of $\mathbf{u}$. Hence,  for every $0\leq t <  T^{\ast },$  it holds that%
\begin{equation}\label{rho-estimate-1}
\|\rho(\cdot, t) \|_{L^{\infty }(\Omega)}=\|\rho _{0}\|_{L^{\infty }(\Omega)}
\end{equation}
and
\be \label{rho-estimate-2}
\|\rho(\cdot, t) - \bar{\rho} \|_{L^{\frac32}(\Omega)} = \| \rho_0 - \bar{\rho} \|_{L^{\frac32}(\Omega)}.
\ee

{\it Step 2. Basic energy estimate.}  The energy estimate can be stated as the following proposition.
\begin{pro}
\label{basic-energy}There
exists some constant $M_1$, which depends only on $\|\sqrt{\rho_0} \mathbf{u}_0\|_{L^2(\Omega)}^2$, $\|\rho_0 - \bar{\rho} \|_{L^{\frac32} (\Omega)}$, $\|\rho_0\|_{L^\infty(\Omega)}$, $\bar{\rho}^{-1}$, such that
\be \label{2ndapri}
\sup_{0< T < T^*} \left\{ \Vert \sqrt{\rho }\mathbf{u}\Vert _{L^{\infty }(0,T;L^{2}(\Omega
))}^{2} + \| \mathbf{u}\|_{L^2(0, T; H^1(\Omega))}^2 \right\} \leq M_1.
\ee
\end{pro}
\begin{proof} Multiplying the first equation of \eqref{NSE} by $\mathbf{u}$ and integrating by
parts over $\Omega $ yield that  for every $0<T<T^{\ast
} $,
\begin{equation}\nonumber
\frac{1}{2} \frac{d}{dt} \int_{\Omega} \rho |\mathbf{u}|^2 \, dx + \int_{\Omega} |\nabla \mathbf{ u}|^2 \, dx = 0.
\end{equation}
Hence,
\begin{equation}
\Vert \sqrt{\rho }\mathbf{u}\Vert _{L^{\infty }(0,T;L^{2}(\Omega
))}^{2}+ 2 \int_{0}^{T}\Vert \nabla \mathbf{u}\Vert _{L^{2}(\Omega
)}^{2}\,dt\leq \int_{\Omega} \rho_0 |\mathbf{u}_0|^2 \, dx .  \label{energy-inequality}
\end{equation}
Moreover, note that
\be \label{L2-H1} \ba
\bar{\rho} \int_{\Omega}|\mathbf{u}|^2 \, dx&  = \int_{\Omega} \rho |\mathbf{u}|^2 \, dx - \int_{\Omega} ( \rho - \bar{\rho} ) |\mathbf{u}|^2 \, dx \\
& \leq \int_{\Omega} \rho |\mathbf{u}|^2 \, dx + \|\rho - \bar{\rho}\|_{L^{\frac32} (\Omega)} \| \mathbf{u}\|_{L^6(\Omega)}^2 \\
& \leq \int_{\Omega} \rho_0 |\mathbf{u}_0|^2 \, dx + C \|\rho_0 - \bar{\rho} \|_{L^{\frac32}(\Omega)} \|\nabla \mathbf{u} \|_{L^2(\Omega)}^2 .
\ea \ee
This,  together with \eqref{energy-inequality}, gives \eqref{2ndapri} with a constant $M_1$ depending only on $\|\sqrt{\rho_0} \mathbf{u}_0\|_{L^2(\Omega)}^2$, $\|\rho_0 - \bar{\rho} \|_{L^{\frac32} (\Omega)}$, and  $\bar{\rho}^{-1}$.
Thus the proof  of Proposition \ref{basic-energy} is completed.
\end{proof}

{\it Step 3.  Estimates for $\Vert \sqrt{\rho }\mathbf{u}_{t}\Vert
_{L^{2}(0,T;L^{2}(\Omega ))}$ and $\Vert \nabla \mathbf{u}\Vert _{L^{\infty
}(0,T;L^{2}(\Omega ))}$.} This is the key step of the whole proof. Higher order estimates of the
density and the velocity can be done in a standard way provided that $\|\mathbf{u}(\cdot, t)%
\|_{H^{1}}$ is uniformly bounded with respect to time. Let
\begin{equation}
\Psi (t)=e+\sup_{0\leq \tau \leq t}\Vert \nabla \mathbf{u}(\cdot ,\tau
)\Vert _{L^{2}(\Omega )}^{2}+\int_{0}^{t}\Vert \sqrt{\rho }\mathbf{u}%
_{t}\Vert _{L^{2}(\Omega )}^{2}\,d\tau ,\ \ \ 0\leq t<T^{\ast }.
\label{H1-quantity}
\end{equation}

To get the $H^{1}$-estimate of $\mathbf{u}$, one can use the  bound of $%
\Vert \mathbf{u}\Vert _{L^{2}(s,T;L^{\infty }(\Omega ))}$.  The key
idea to get the bound of $\Vert
\mathbf{u}\Vert _{L^{2}(s,T;L^{\infty }(\Omega ))}$ is that
an axisymmetric function can be regarded as a
function of two variables in some sense.
\begin{lemma}
\label{Log-inequality} There
exists a constant $C_{2}$ independent of $s$ and  $ T$, such that for every $0\leq s<T<T^{\ast }$,
\begin{equation}
\int_{s}^{T}\Vert \mathbf{u}\Vert _{L^{\infty }(\Omega )}^{2}\,d\tau \leq
C_{2}\Vert \nabla \mathbf{u}\Vert _{L^{2}(s,T; L^2(\Omega ))}^{2}  \ln \Psi
(T)+ C_2.  \label{L2 infty}
\end{equation}
\end{lemma}

\begin{proof}
Denote $D_{2}=(1,+\infty )\times \mathbb{R}$. Since $\mathbf{u}(x,y,z,t)$ is
axisymmetric, $\mathbf{u}$ can also be considered as a function defined on $%
D_{2}\times \lbrack 0,T^{\ast })$,
\begin{equation}
\mathbf{u}(x,y,z,t)=u^{r}(r,z,t){\boldsymbol{e}}_{r}+u^{\theta }(r,z,t){%
\boldsymbol{e}}_{\theta }+u^{z}(r,z,t){\boldsymbol{e}}_{z}.  \notag
\end{equation}%
Let $\tilde{\nabla}=(\partial _{r},\partial _{z})$ be the two-dimensional
gradient operator, and $\tilde{W}^{1,6}(D_{2})$ be the Sobolev space defined
on $D_{2}$. By Lemma \ref{Brezis-Inequality},%
\begin{equation}\label{41}
\begin{aligned}
&\int_{s}^{T}\Vert \mathbf{u}\Vert _{L^{\infty }(\Omega )}^{2}\,d\tau
 \\
\leq & C\int_{s}^{T}\left( \Vert u^{r}\Vert _{L^{\infty }(D_{2})}+\Vert
u^{\theta }\Vert _{L^{\infty }(D_{2})}+\Vert u^{z}\Vert _{L^{\infty
}(D_{2})}\right) ^{2}\,d\tau   \\
\leq & C\Vert \tilde{\nabla}(u^{r},u^{\theta },u^{z})\Vert
_{L^{2}(s,T;L^{2}(D_{2}))}^{2} \ln \left( e + \Vert (u^{r},u^{\theta
},u^{z})\Vert _{L^{2}(s,T;\tilde{W}^{1,6}(D_{2}))}\right)+C  .
\end{aligned}
\end{equation}

Note that
\begin{eqnarray}
&&\partial _{r}u^{r} =\partial _{1}u^{1}\cos ^{2}\theta +\partial
_{2}u^{1}\sin \theta \cos \theta +\partial _{1}u^{2}\cos \theta \sin \theta
+\partial _{2}u^{2}\sin ^{2}\theta ,  \label{42_1} \\
&&\partial _{r}u^{\theta } =-\partial _{1}u^{1}\cos \theta \sin \theta
-\partial _{2}u^{1}\sin ^{2}\theta +\partial _{1}u^{2}\cos ^{2}\theta
+\partial _{2}u^{2}\sin \theta \cos \theta ,  \label{42_2}
\end{eqnarray}%
and%
\begin{equation}
\partial _{r}u^{z}=\partial _{1}u^{3}\cos \theta +\partial _{2}u^{3}\sin
\theta .  \label{43}
\end{equation}%
Hence one has
\begin{equation}
\Vert \tilde{\nabla}(u^{r},u^{\theta },u^{z})\Vert
_{L^{2}(s,T;L^{2}(D_{2}))}\leq C\Vert \nabla \mathbf{u}\Vert
_{L^{2}(s,T;L^{2}(\Omega ))} \label{44}
\end{equation}%
and
\begin{equation}
\Vert (u^{r},u^{\theta },u^{z})\Vert _{L^{2}(s,T;\tilde{W}%
^{1,6}(D_{2}))}\leq C\Vert \mathbf{u}\Vert _{L^{2}(s,T;W^{1,6}(\Omega ))}.
\label{45}
\end{equation}
Combining \eqref{41} and \eqref{44}-\eqref{45} together gives
\be \label{45-1}
\int_s^T \| \mathbf{u}\|_{L^\infty(\Omega)}^2 \, d\tau
\leq C \| \nabla \mathbf{u} \|_{L^2(s, T; L^2(\Omega))}^2 \ln \left( e  +
\| \mathbf{u} \|_{L^2(s, T; W^{1, 6}(\Omega))}      \right) + C .
\ee

It follows from Sobolev embedding inequality and Lemma \ref{lemma lam2} that one has
\begin{equation} \label{ineqfrom2.4}
\begin{array}{ll}
& \Vert \mathbf{u}\Vert _{L^{2}(s,T;W^{1,6}(\Omega ))}^{2}\leq C\left( \Vert
\nabla \mathbf{u}\Vert _{L^{2}(s,T;L^{2}(\Omega ))}^{2}+\Vert \nabla ^{2}%
\mathbf{u}\Vert _{L^{2}(s,T;L^{2}(\Omega ))}^{2}\right) \\
 \leq & C\left( \Vert \mathbf{u}\Vert _{L^{2}(s,T;H^{1}(\Omega ))}^{2}+\Vert
\sqrt{\rho }\mathbf{u}_{t}\Vert _{L^{2}(s,T;L^{2}(\Omega ))}^{2}+\Vert (\rho
\mathbf{u}\cdot \nabla )\mathbf{u}\Vert _{L^{2}(s,T;L^{2}(\Omega)
)}^{2}\right) \\[0.12in]
 \leq & C\left( \Vert \mathbf{u}\Vert _{L^{2}(s,T;H^{1}(\Omega ))}^{2}+\Vert
\sqrt{\rho }\mathbf{u}_{t}\Vert _{L^{2}(s,T;L^{2}(\Omega ))}^{2}+\Vert
\mathbf{u}\Vert _{L^{2}(s,T;L^{\infty }(\Omega ))}^{2}\Vert \nabla
\mathbf{u}\Vert _{L^{\infty }(s,T;L^{2}(\Omega ))}^{2}\right) \\[0.12in]
 \leq & C\left[ \Vert \mathbf{u}\Vert _{L^{2}(s,T;H^{1}(\Omega ))}^{2}+\Psi
(T)+\Vert \mathbf{u}\Vert _{L^{2}(s,T;L^{\infty }(\Omega ))}^{2} \Psi
(T)\right] .%
\end{array}
\end{equation}
This,  together with the estimate \eqref{45-1}, implies that
\begin{equation}
\begin{array}{ll}
& \Vert \mathbf{u}\Vert _{L^{2}(s,T;L^{\infty }(\Omega ))}^{2} \\[0.12in]
\leq & C \Vert \nabla \mathbf{u}\Vert _{L^{2}(s,T; L^2(\Omega ))}^{2} \ln  \left[ \| \mathbf{u} \|_{L^2(s, T; H^1(\Omega) )}^2
+ \Psi(T) + \|\mathbf{u} \|_{L^2(s, T; L^\infty(\Omega))} \Psi(T)
\right] + C  \\
\leq & C\Vert \nabla \mathbf{u}\Vert _{L^{2}(s,T; L^2(\Omega ))}^{2} \ln \Psi
(T)+ C_1 \ln \left( 1+\Vert \mathbf{u}\Vert _{L^{2}(s,T;L^{\infty }(\Omega
))}^{2}\right)+C  .%
\end{array}
\label{50}
\end{equation}%
Choose $N_{1}$ sufficiently large such that%
\begin{equation}
C_{1}\ln (1+\gamma )\leq \frac{1}{2}\gamma ,\ \ \mbox{for}\ \gamma \geq
N_{1}.  \notag
\end{equation}%
Then one has%
\begin{equation}
C_{1}\ln \left( 1+\Vert \mathbf{u}\Vert _{L^{2}(s,T;L^{\infty }(\Omega
))}^{2}\right) \leq \frac{1}{2}\Vert \mathbf{u}\Vert _{L^{2}(s,T;L^{\infty
}(\Omega ))}^{2} + \frac{1}{2} N_1 .  \label{51}
\end{equation}%
Combining  (\ref{50}) and \eqref{51} yields
\begin{equation}
\int_{s}^{T}\Vert \mathbf{u}\Vert _{L^{\infty }(\Omega )}^{2}\,d\tau \leq
C_{2}\Vert \nabla \mathbf{u}\Vert _{L^{2}(s,T; L^2(\Omega ))}^{2} \ln \Psi
(T)+ C_{2}.  \label{52}
\end{equation}%
This finishes the proof of the proposition.
\end{proof}

With the estimate (\ref{L2 infty}) at hand, one can prove the following estimate.

\begin{pro}
\label{H1estimate} It holds that
\begin{equation}
\sup_{0<T<T^{\ast }}\left\{ \Vert \nabla \mathbf{u}\Vert _{L^{2}(\Omega
)}^{2}+\int_{0}^{T}\Vert \sqrt{\rho }\mathbf{u}_{t}\Vert _{L^{2}(\Omega
)}^{2}\,dt + \int_0^T \|\nabla \mathbf{u} \|_{H^1(\Omega)}^2 \, dt \right\} <+\infty .  \label{esti 1}
\end{equation}
\end{pro}

\begin{proof}
Multiplying the first equation of \eqref{NSE} by $\partial _{t}\mathbf{u}$
and integrating over $\Omega $ lead to%
\begin{equation}
\frac{1}{2}\frac{d}{dt}\int_{\Omega }|\nabla \mathbf{u}|^{2}\,dx+\int_{%
\Omega }\rho |\mathbf{u}_{t}|^{2}\,dx=-\int_{\Omega }(\rho \mathbf{u}\cdot
\nabla )\mathbf{u}\cdot \mathbf{u}_{t}\,dx.  \label{55}
\end{equation}%
By H\"{o}lder inequality and Young's inequality,%
\begin{equation} \label{56}
\begin{aligned}
\left\vert \int_{\Omega }(\rho \mathbf{u}\cdot \nabla )\mathbf{u}\cdot
\mathbf{u}_{t}\,dx\right\vert \leq &C \|\sqrt{\rho }\mathbf{u}%
_{t}\|_{L^{2}(\Omega )} \Vert \mathbf{u}\Vert _{L^{\infty }(\Omega
)} \Vert \nabla \mathbf{u}\Vert _{L^{2}(\Omega )} \\
\leq &\frac{1}{2}\Vert \sqrt{\rho }\mathbf{u}_{t}\Vert _{L^{2}(\Omega
)}^{2}+C\Vert \mathbf{u}\Vert _{L^{\infty }(\Omega )}^{2}  \Vert \nabla
\mathbf{u}\Vert _{L^{2}(\Omega )}^{2}.
\end{aligned}
\end{equation}%
Substituting \eqref{56} into \eqref{55} gives
\begin{equation}
\frac{d}{dt}\int_{\Omega }|\nabla \mathbf{u}|^{2}\,dx+\int_{\Omega }|\sqrt{%
\rho }\mathbf{u}_{t}|^{2}\,dx\leq C\Vert \mathbf{u}\Vert _{L^{\infty
}(\Omega )}^{2}  \Vert \nabla \mathbf{u}\Vert _{L^{2}(\Omega )}^{2}.
\label{57}
\end{equation}%
Hence, for every $0\leq
s<T<T^{\ast }$,
\begin{equation}
\Vert \nabla \mathbf{u}(T)\Vert _{L^{2}(\Omega )}^{2}+\int_{s}^{T}\Vert
\sqrt{\rho }\mathbf{u}_{t}\Vert _{L^{2}(\Omega )}^{2}\,d\tau \leq \Vert
\nabla \mathbf{u}(s)\Vert _{L^{2}(\Omega )}^{2} \exp \left\{ C
\int_{s}^{T}\Vert \mathbf{u}\Vert _{L^{\infty }(\Omega )}^{2}d\tau \right\} .
\label{58}
\end{equation}%
Consequently,
\begin{equation}
\ba
\displaystyle \Psi (T) & \leq \Psi (s)  \exp \left\{ C \int_{s}^{T}  \Vert
\mathbf{u}\Vert _{L^{\infty }(\Omega )}^{2}\,d\tau \right\}
+\int_{0}^{s}\Vert \sqrt{\rho }\mathbf{u}_{t}\Vert _{L^{2}(\Omega
)}^{2}\,d\tau + \Psi(s) \\
& \leq 3\Psi (s)\exp \left\{ C \int_{s}^{T}\Vert \mathbf{u}\Vert
_{L^{\infty }}^{2}\,d\tau \right\} .%
\ea
\label{59}
\end{equation}%
This, together with  Proposition \ref{Log-inequality}, gives
\begin{equation}
\begin{array}{ll}
\Psi (T) & \leq 3\Psi (s)  \exp \left\{ C_3 \Vert \nabla \mathbf{u}%
\Vert _{L^{2}(s,T;L^{2}(\Omega )}^{2} \ln \Psi (T)+ C_3 \right\} \\%
[3mm]
& \leq C\Psi (s)\Psi (T)^{C_3 \Vert \nabla \mathbf{u}\Vert
_{L^{2}(s,T; L^2(\Omega ))}^{2}}.%
\end{array}
\label{65}
\end{equation}
Recalling the basic energy inequality, one can choose some $s_{0}$ close enough to $T^{\ast }$, such that
\begin{equation}
C_3 \Vert \nabla \mathbf{u}\Vert
_{L^{2}(s_{0},T; L^2(\Omega ))}^{2}\leq \frac{1}{2}.  \label{66}
\end{equation}%
Therefore,  for every $s_{0}<T<T^{\ast }$, one has
\begin{equation}
\Psi (T)\leq C\Psi (s_{0})^{2}<+\infty.   \label{67}
\end{equation}%

Combining the estimates in Proposition \ref{Log-inequality} and Proposition %
\ref{H1estimate} yields
\begin{equation}
\sup_{0<T<T^{\ast }}\int_{0}^{T}\Vert \mathbf{u}\Vert _{L^{\infty }(\Omega
)}^{2}\,dt<+\infty .  \label{75}
\end{equation}%
For every $T \in (0,T^{\ast })$, it follows from the inequality %
\eqref{ineqfrom2.4} that
\begin{equation}
\int_{0}^{T}\Vert \nabla ^{2}\mathbf{u}\Vert _{L^{2}(\Omega )}^{2}\,dt\leq
C\left( \Vert \mathbf{u}\Vert _{L^{2}(0,T;H^{1}(\Omega ))}^{2}+\Psi
(T)+\Vert \mathbf{u}\Vert _{L^{2}(s,T;L^{\infty }(\Omega ))}^{2}  \Psi
(T)\right) <+\infty .  \notag
\end{equation}%
This finishes the proof of the proposition.
\end{proof}

{\it Step 4. Estimates for $\Vert \sqrt{\rho }\mathbf{u}_{t}\Vert
_{L^{\infty }(0,T;L^2(\Omega) )}$ and $\Vert \nabla \mathbf{u}_{t}\Vert _{L^{2
}(0,T;L^{2}(\Omega))}$.} These estimates can be stated as the following proposition.

\begin{pro}
\label{step 4 lemma} Suppose that $(\rho ,\mathbf{u})$ is a local strong solution to the problem \eqref{NS}--\eqref{IBVP}, it
holds that%
\begin{equation}
\sup_{0<T<T^{\ast }}\left\{ \Vert \sqrt{\rho }\mathbf{u}_{t}\Vert
_{L^{2}(\Omega )}^{2} + \|\mathbf{u}\|_{H^2(\Omega)}^2 +\int_{0}^{T}\Vert \nabla \mathbf{u}_{t}\Vert
_{L^{2}(\Omega )}^{2}\,dt\right\} <+\infty .  \label{esti 2}
\end{equation}
\end{pro}

\begin{proof}
Taking the derivative of the first equation in (\ref{NSE}) with respect to $%
t,$ gives
\begin{equation}\label{time deri}
\rho \mathbf{u}_{tt}+(\rho \mathbf{u}\cdot \nabla )\mathbf{u}_{t}-\Delta
\mathbf{u}_{t}+\nabla P_{t}  \notag =-\rho _{t}\mathbf{u}_{t}-(\rho _{t}\mathbf{u}\cdot \nabla )\mathbf{u}%
-(\rho \mathbf{u}_{t}\cdot \nabla )\mathbf{u}.
\end{equation}%
Multiplying (\ref{time deri}) by $\mathbf{u}_{t}$ and integrating over $%
\Omega$ yield
\begin{equation}\label{3 terms}
\begin{aligned}
&\frac{1}{2}\frac{d}{dt}\int_{\Omega }\rho |\mathbf{u}_{t}|^{2}\,dx+\int_{%
\Omega }|\nabla \mathbf{u}_{t}|^{2}\,dx  \\
=&-\int_{\Omega }\rho _{t}|\mathbf{u}_{t}|^{2}\,dx-\int_{\Omega
}(\rho _{t}\mathbf{u}\cdot \nabla )\mathbf{u}\cdot \mathbf{u}%
_{t}\,dx-\int_{\Omega }(\rho \mathbf{u}_{t}\cdot \nabla )\mathbf{u}\cdot
\mathbf{u}_{t}\text{ }dx.
\end{aligned}
\end{equation}%
One can estimate the three terms on the right-hand side of (\ref{3 terms}) one by one. Taking the second equation of (\ref{NSE}) into account, and using Gagliardo-Nirenberg inequality yield%
\begin{equation}\label{3-1}
\begin{aligned}
-\int_{\Omega }\rho _{t}|\mathbf{u}_{t}|^{2}\,dx = &
\int_{\Omega } {\rm div}~(\rho \mathbf{u} )  |\mathbf{u}_{t}|^{2}\,dx
 \\
=&-2 \int_{\Omega } \rho \mathbf{u} \cdot \nabla \mathbf{u}_{t}\cdot \mathbf{u}%
_{t}\,dx   \\
 \leq & C \| \mathbf{u} \|_{L^6(\Omega)} \|\nabla \mathbf{u}_t\|_{L^2(\Omega)} \|\rho \mathbf{u}_t\|_{L^3(\Omega)} \\
\leq &C \| \nabla \mathbf{u}\|_{L^{2}(\Omega )}\|\nabla \mathbf{u}%
_{t}\|_{L^{2}(\Omega )}^{\frac32} \|\sqrt{\rho }\mathbf{u}_{t}\|_{L^{2}(\Omega
)}^{\frac12}   \\
\leq &\frac{1}{16} \|\nabla \mathbf{u}_{t}\|_{L^{2}(\Omega )}^{2}+C\|\sqrt{%
\rho }\mathbf{u}_{t}\|_{L^{2}(\Omega )}^{2}\|\nabla \mathbf{u}%
\|_{L^{2}(\Omega )}^{4}.
\end{aligned}
\end{equation}%
For the second term of (\ref{3 terms}), one has%
\begin{equation} \label{3-2}
\begin{aligned}
-\int_{\Omega }(\rho _{t}\mathbf{u}\cdot \nabla )\mathbf{u}\cdot \mathbf{u}%
_{t}\,dx =&\int_{\Omega }{\rm div}~( \rho \mathbf{u} )(\mathbf{u}\cdot
\nabla )\mathbf{u}\cdot \mathbf{u}_{t}\,dx  \notag \\
=&-\int_{\Omega }(\rho \mathbf{u} )\cdot \nabla \lbrack (\mathbf{u}\cdot
\nabla )\mathbf{u}\cdot \mathbf{u}_{t}]\,dx  \notag \\
\leq &\int_{\Omega }|\rho \mathbf{u}_{t}||\mathbf{u}||\nabla \mathbf{u}%
|^{2}\,dx+\int_{\Omega }|\rho \mathbf{u}_{t}||\mathbf{u}|^{2}|\nabla ^{2}%
\mathbf{u}|\,dx  \notag \\
&+\int_{\Omega }\rho |\mathbf{u}|^{2}|\nabla \mathbf{u}||\nabla \mathbf{u}%
_{t}|\,dx.
\end{aligned}
\end{equation}%
By Sobolev inequality, one has%
\begin{equation}\label{3-3}
\begin{aligned}
\int_{\Omega }|\rho \mathbf{u}_{t}||\mathbf{u}||\nabla \mathbf{u}|^{2}\,dx
\leq & \|\rho \|_{L^{\infty }(\Omega) } \|\mathbf{u}_{t}\|_{L^{6}(\Omega)}  \|\mathbf{u}\|_{L^{6}(\Omega) }  \|\nabla \mathbf{u}\|_{L^3(\Omega)}^2   \\
\leq & C \|\nabla \mathbf{u}\|_{L^{2}(\Omega)}^{2} \|\nabla \mathbf{u}%
_{t}\|_{L^{2}(\Omega) } \|\nabla \mathbf{u}\|_{H^{1}(\Omega)}   \\
\leq &\frac{1}{16}\|\nabla \mathbf{u}_{t}\|_{L^{2}(\Omega )}^{2}+C\|\nabla
\mathbf{u}\|_{L^{2}(\Omega )}^{4}\|\nabla \mathbf{u}\|_{H^{1}(\Omega )}^{2}.
\end{aligned}
\end{equation}%
Using H\"{o}lder and Young's inequalities gives%
\begin{eqnarray}
\int_{\Omega }|\rho \mathbf{u}_{t}||\mathbf{u}|^{2}|\nabla ^{2}\mathbf{u}%
|\,dx &\leq &C \|\rho \|_{L^{\infty }(\Omega )} \|\mathbf{u}%
_{t}\|_{L^{6}(\Omega )} \|\mathbf{u}\|_{L^{6}(\Omega )}^{2} \|\nabla ^{2}%
\mathbf{u}\|_{L^{2}(\Omega )}  \notag \\
&\leq &C \|\nabla \mathbf{u}_{t}\|_{L^{2}(\Omega )}\|\nabla \mathbf{u}%
\|_{L^{2}(\Omega )}^{2} \|\nabla \mathbf{u}\|_{H^{1}(\Omega )}  \notag \\
&\leq &\frac{1}{16}\|\nabla \mathbf{u}_{t}\|_{L^{2}(\Omega )}^{2}+C\|\nabla
\mathbf{u}\|_{L^{2}(\Omega )}^{4}\|\nabla \mathbf{u}\|_{H^{1}(\Omega )}^{2}.
\label{3-4}
\end{eqnarray}%
Similarly, one has%
\begin{equation}
\int_{\Omega }\rho |\mathbf{u}|^{2}|\nabla \mathbf{u}| |\nabla \mathbf{u}%
_{t}|\,dx\leq \frac{1}{16} \|\nabla \mathbf{u}_{t}\|_{L^{2}(\Omega
)}^{2}+C \|\nabla \mathbf{u}\|_{L^{2}(\Omega )}^{4} \|\nabla \mathbf{u}%
\|_{H^{1}(\Omega )}^{2}.  \label{3-5}
\end{equation}%
For the third term on the right-hand side of (\ref{3 terms}), similar to the
estimate in (\ref{3-1}), one has%
\begin{equation}
-\int_{\Omega }(\rho \mathbf{u}_{t}\cdot \nabla )\mathbf{u}\cdot \mathbf{u}%
_{t}\, dx \leq \frac{1}{16} \|\nabla \mathbf{u}_{t}\|_{L^{2}(\Omega
)}^{2}+C\|\sqrt{\rho }\mathbf{u}_{t}\|_{L^{2}(\Omega )}^{2}\|\nabla \mathbf{u%
}\|_{L^{2}(\Omega )}^{4}.  \label{3-6}
\end{equation}%
Therefore, collecting all the estimates (\ref{3 terms})-(\ref{3-6}) and
taking (\ref{esti 1}) into account yield
\begin{eqnarray*}
&& \frac{1}{2}\frac{d}{dt}\int_{\Omega }|\sqrt{\rho }\mathbf{u}_{t}|^{2}\,dx+%
\frac{1}{4}\int_{\Omega }|\nabla \mathbf{u}_{t}|^{2}\,dx \\
&\leq &C \|\nabla \mathbf{u}\|_{L^{2}(\Omega )}^{4} \|\sqrt{%
\rho }\mathbf{u}_{t}\|_{L^{2}(\Omega )}^{2}+C\|\nabla \mathbf{u}%
\|_{H^{1}(\Omega )}^{2} \|\nabla \mathbf{u}\|_{L^{2}(\Omega )}^{4}.
\end{eqnarray*}%
This, together with Gronwall's inequality, shows
\be \label{3-6-1}
\sup_{0\leq T < T^*} \left[  \int_{\Omega} |\sqrt{\rho}\mathbf{u}_t|^2 \, dx + \int_0^T \int_{\Omega} |\nabla \mathbf{u}_t|^2 \, dx dt           \right]
< + \infty.
\ee
Furthermore, the second equation of (\ref{NSE}), together with
Lemma \ref{lemma lam2}, gives
\begin{eqnarray*}
\| \mathbf{u} \|_{H^{2}(\Omega )} &\leq &C \|\mathbf{u}||_{H^{1}(\Omega
)}+C \|\rho \mathbf{u}_{t} \|_{L^{2}(\Omega )}+C \|\rho \mathbf{u}\cdot \nabla
\mathbf{u}\|_{L^{2}(\Omega )} \\
&\leq &C \|\mathbf{u} \|_{H^{1}(\Omega )} + C \|\sqrt{\rho }\mathbf{u}%
_{t}\|_{L^{2}(\Omega )} + C\|\rho \|_{L^{\infty }(\Omega )}\|\mathbf{u}\cdot
\nabla \mathbf{u} \|_{L^{2}(\Omega )} \\
&\leq &C \|\mathbf{u}\|_{H^{1}(\Omega )}+C \|\sqrt{\rho }\mathbf{u}%
_{t}\|_{L^{2}(\Omega )} + C \|\mathbf{u} \|_{L^{6}(\Omega )} \|\nabla \mathbf{u}%
\|_{L^{3}(\Omega )} \\
&\leq &C \|\mathbf{u} \|_{H^{1}(\Omega )}+C\|\sqrt{\rho }\mathbf{u}%
_{t}\|_{L^{2}(\Omega )}+C\|\nabla \mathbf{u}\|_{L^{2}(\Omega
)}^{\frac32}\|\nabla \mathbf{u}\|_{H^1 (\Omega )}^{\frac12}.
\end{eqnarray*}%
It follows from Young's inequality and  the bounds for  $\|\mathbf{u}%
\|_{L^{\infty }(0,T;H^{1}(\Omega ))}$ and  $\| \sqrt{\rho }\mathbf{u}%
_{t} \|_{L^{\infty }(0,T;L^{2}(\Omega ))}$ that
\begin{equation*}
\sup_{0\leq T < T^*} \| \mathbf{u}(\cdot, T)\|_{H^2(\Omega)} < + \infty.
\end{equation*}
Hence, the proof of Proposition \ref{step 4 lemma} is completed.
\end{proof}

{\it Step 5. Estimates for $\Vert \nabla \rho \Vert _{L^{\infty
}(0,T;H^{1}(\Omega ))}$ and $\Vert \rho_t \Vert _{L^{\infty}(0,T;H^1 (\Omega
))}.$} These estimates can be summarized as follows.

\begin{pro}
\label{step 5 lemma} It holds that
\begin{equation}
\sup_{0<T<T^{\ast }} \left( \| \nabla \rho \|_{L^{\infty }(0,T;H^{1}(\Omega
))} + \|\rho_t\|_{L^\infty(0, T; H^1(\Omega))} \right)
<+\infty .  \label{esti 3}
\end{equation}
\end{pro}

\begin{proof}
Differentiating the second equation of (\ref{NSE}) with respect to $x_j$ ($j=1$, $2$, $3$) yields
\begin{equation*}
(\rho _{x_{j}})_{t}+\mathbf{u}\cdot \nabla \rho _{x_{j}}=-\mathbf{u}%
_{x_{j}}\cdot \nabla \rho .
\end{equation*}%
Multiplying the resulting equation by $\rho _{x_{j}},$ integrating over $%
\Omega $, and summing up give
\begin{equation*}
\frac{d}{dt}\int_{\Omega }|\nabla \rho |^{2}\,dx\leq C\int_{\Omega }|\nabla
\mathbf{u}||\nabla \rho |^{2}\,dx\leq C \| \nabla \mathbf{u} \|_{L^{\infty }(\Omega )} \|\nabla \rho \|_{L^{2}(\Omega )}^{2}.
\end{equation*}%
A similar argument shows that%
\begin{eqnarray*}
\frac{d}{dt}\int_{\Omega }|\nabla ^{2}\rho |^{2}\,dx &\leq &C\int_{\Omega
}|\nabla \mathbf{u}||\nabla ^{2}\rho |^{2}\,dx+ \int_{\Omega }|\nabla ^{2}%
\mathbf{u}||\nabla \rho ||\nabla ^{2}\rho |\,dx \\
&\leq &C \| \nabla \mathbf{u}\|_{L^{\infty }(\Omega )} \|\nabla
^{2}\rho \|_{L^{2}(\Omega )}^{2}+ \|\nabla ^{2}\mathbf{u} \|_{L^{6}(\Omega
)}\|\nabla \rho \|_{L^{3}(\Omega )} \|\nabla ^{2}\rho \|_{L^{2}(\Omega )}.
\end{eqnarray*}%
It follows from  Sobolev embedding inequality and Gronwall's inequality that 
\begin{equation*}
\Vert \nabla \rho \Vert _{H^{1}(\Omega )}^{2}\leq C\Vert \nabla \rho
_{0}\Vert _{H^{1}(\Omega )}^{2}\exp \left( C\int_{0}^{T}\Vert \nabla \mathbf{%
u}\Vert _{W^{1,6}(\Omega )}  \,dt\right) .
\end{equation*}%
Herein, by Lemma \ref{lemma lam2},%
\begin{eqnarray*}
 \| \nabla \mathbf{u} \|_{W^{1,6}(\Omega )} & \leq & C \|\mathbf{u}%
\|_{W^{1,6} (\Omega )}+C\|\rho \mathbf{u}_{t}\|_{L^{6}(\Omega )}+C \|\rho
\mathbf{u}\cdot \nabla \mathbf{u}\|_{L^{6}(\Omega )} \\
&\leq &C \|\mathbf{u}\|_{H^{2}(\Omega ) } + C \|\nabla \mathbf{u}%
_{t}\|_{L^{2}(\Omega )}+C \|\mathbf{u}\|_{L^{\infty }(\Omega )} \|\nabla
\mathbf{u}\|_{L^{6}(\Omega )} \\
&\leq &C \|\mathbf{u}\|_{H^{2}(\Omega )}+C \|\nabla \mathbf{u}%
_{t}\|_{L^{2}(\Omega )}+C \|\mathbf{u} \|_{L^{\infty }(\Omega)} \|\nabla \mathbf{u}%
\|_{H^1 (\Omega )}.
\end{eqnarray*}%
By Proposition \ref{step 4 lemma}, one has
\begin{equation*}
\sup_{0\leq T < T^* } \|\nabla \rho \|_{H^1(\Omega)} < +\infty .
\end{equation*}%
Moreover, according to the second equation of \eqref{NSE} and Proposition \ref{step 4 lemma}, it holds that
\be \nonumber
\sup_{0\leq T < T^* }\|\rho_t\|_{H^1(\Omega)} \leq \sup_{0\leq T < T^* } \|\mathbf{u} \cdot \nabla \rho \|_{H^1(\Omega)} < + \infty.
\ee
This finishes the proof of Proposition \ref{step 5 lemma}.
\end{proof}

Combining all the estimates in (\ref{esti 1}), (\ref{esti 2}) and (\ref{esti
3}) yields (\ref{main esti}). This, together with Theorem \ref
{local}, shows that the local strong solution $(\rho, \mathbf{u})$ does not blow up at $T^*$. Hence, $(\rho, \mathbf{u})$ is in fact a global solution so that the proof of Theorem \ref{main-result} is completed.


\section{Local well-posedness of the strong solutions in the exterior domains\label{local result}}
The local existence and uniqueness of strong solutions for inhomogeneous Navier-Stokes system in bounded domains has been proved in \cite{Choe} by Galerkin approximation.  In this section, the  local existence of strong solutions in the exterior domain   is established as a limit of local strong solutions in a sequence of bounded domains constructed in \cite{Choe}.   To prove the convergence of approximate solutions, one of the key observations is that the lifespan of each local strong solution depends only on the $C^3$-regularity of the domain, $\|\rho_0\|_{L^\infty(\Omega)}$, $\|\rho_0 - \bar{\rho} \|_{L^{\frac32}(\Omega)}$, $\bar{\rho}^{-1}$, $\|\mathbf{u}_0\|_{H^1(\Omega)}$, but is independent of the size of $\partial \Omega$ and $\Omega$.  This observation relies heavily on the uniform estimates for the Stokes problem (cf. Lemma \ref{lemma lam2}).

In this section, assume that $\tilde{\Omega}$ is a bounded domain of $\mathbb{R}^3$, the boundary of which is uniformly of class $C^3$. And for simplicity of notations, denote
$L^r = L^r(\tilde{\Omega})$, $W^{k, p} = W^{k, p}(\tilde{\Omega})$, $H^k = H^k(\tilde{\Omega})$, and
\be \nonumber
\int f \, dx = \int_{\tilde{\Omega}} f \, dx.
\ee

Given initial data $(\rho_0, \mathbf{u}_0) \in L^\infty \times (H^2 \cap H_{0, \sigma}^1 )$, suppose that $(\rho_0, \mathbf{u}_0)$ satisfy
\be \label{compatibility-new}
\rho_0 \geq 0, \ \ \ \rho_0 - \bar{\rho} \in H^2 \cap L^{\frac32}, \ \ \
- \mu \Delta \mathbf{u}_0 + \nabla P_0 = \rho_0^{\frac12} \mathbf{g},\ \ \mbox{in}\ \tilde{\Omega},
\ee
with some $(P_0, \mathbf{g})$ belonging to $D^{1,2} \times L^2$. The following is the local existence result proved in \cite{Choe}.
\begin{lemma}\label{localexistence}
Assume that the initial data $(\rho_0, \mathbf{u}_0)$ satisfies $\rho_0 \in L^\infty$, $\mathbf{u}_0 \in H_{0, \sigma}^1 \cap H^2$ and \eqref{compatibility-new}. There exists a positive time $T_0$ and a unique strong solution $(\rho, \mathbf{u})$ to the initial boundary value problem \eqref{NS}--\eqref{IBVP} such that
\be \nonumber \begin{array}{c}
\rho- \bar{\rho} \in C ([0, T_0); L^{\frac32} \cap H^2),\ \ \ \mathbf{u} \in C ([0, T_0); H_{0, \sigma}^1 )\cap L^\infty(0, T_0; H^2), \\[3mm]
\rho_t \in L^\infty(0, T_0; H^1),\ \ \ \sqrt{\rho} \mathbf{u}_t \in L^\infty(0, T_0; L^2),\ \ \ \nabla \mathbf{u}_t \in L^2(0, T_0; L^2).
\end{array}
\ee
\end{lemma}

Next, one can prove that $T_0$ has a uniform lower bound, which depends only on $\|\rho_0 - \bar{\rho} \|_{L^{\frac32} } $, $\|\rho_0 \|_{L^\infty}$, $\bar{\rho}^{-1}$, $\| \mathbf{u}_0\|_{H^1}$ and the $C^3$-regularity of $\partial \tilde{\Omega}$. Note that the lower bound does not depend on the size of $\partial \tilde{\Omega}$ or $\tilde{\Omega}$.  And also  some uniform estimates for solutions independent of the size of $\partial \tilde{\Omega}$ and $\tilde{\Omega}$, can be established.

\begin{lemma}\label{uniform-1}
For every $0 \leq t < T_0$, it holds that
\be \label{uniform-001}
\| \rho(\cdot, t) - \bar{\rho} \|_{L^{\frac32}} = \|\rho_0 - \bar{\rho} \|_{L^{\frac32}},\ \ \ \mbox{and}\ \ \ \|\rho(\cdot, t) \|_{L^\infty} = \|\rho_0 \|_{L^\infty }.
\ee
\end{lemma}

\begin{lemma} \label{uniform-2}
There exists some constant $C$, which depends on  $\|\rho_0 - \bar{\rho}\|_{L^{\frac32}}$, $\|\sqrt{\rho_0} \mathbf{u}_0 \|_{L^2}$, $\bar{\rho}^{-1}$, such that for every $0< T< T_0$,
\be \label{uniform-002}
\|\sqrt{\rho } \mathbf{u}\|_{L^\infty(0, T; L^2)} + \|\mathbf{u}\|_{L^2(0, T; H^1)} \leq C .
\ee
\end{lemma}
The proofs of Lemmas \ref{uniform-1}--\ref{uniform-2} are the same as those for \eqref{rho-estimate-1}--\eqref{2ndapri}.

\begin{lemma}\label{uniform-3}
There exist some positive constants $C$ and $T_0^*$, which depend on $\|\rho_0\|_{L^\infty}$, $\|\rho_0 - \bar{\rho}\|_{L^{\frac32}}$, $\| \mathbf{u}_0 \|_{H^1}$, $\bar{\rho}^{-1}$ and $C^3$-regularity of $\partial \tilde{\Omega}$, such that for every $0< T< T_0^*$,
\be \label{uniform-005}
\sup_{0\leq t < T}\| \mathbf{u} \|_{H^1}^2 + \int_0^T \|\sqrt{\rho} \mathbf{u}_t \|_{L^2}^2 \, dt + \int_0^T \|\mathbf{u} \|_{H^2}^2 \, dt \leq C .
\ee
\end{lemma}

\begin{proof}
Multiplying the first equation of \eqref{NSE} by $\partial_t \mathbf{u}$ and integrating yield
\be \nonumber
\frac12 \frac{d}{dt} \int  |\nabla \mathbf{u} |^2 \, dx + \int  \rho |\mathbf{u}_t|^2 \, dx =
- \int ( \rho \mathbf{u} \cdot \nabla ) \mathbf{u} \cdot \mathbf{u}_t \, dx .
\ee
By Sobolev embedding inequality and Young's inequality,
\be \label{uniform-007} \ba
\left| \int ( \rho \mathbf{u} \cdot \nabla ) \mathbf{u} \cdot \mathbf{u}_t \, dx \right|
& \leq \|\sqrt{\rho } \|_{L^\infty } \|\sqrt{\rho} \mathbf{u}_t\|_{L^2 } \| \mathbf{u}\|_{L^\infty} \| \nabla \mathbf{u} \|_{L^2} \\
& \leq C \|\sqrt{\rho} \mathbf{u}_t\|_{L^2 } \| \mathbf{u}\|_{H^1}^{\frac32} \|\mathbf{u}\|_{H^2}^{\frac12}\\
& \leq \frac12 \|\sqrt{\rho } \mathbf{u}_t \|_{L^2 }^2 + C \|\mathbf{u}\|_{H^1}^3 \| \mathbf{u}\|_{H^2}.
\ea \ee
Herein, using Lemma \ref{lemma lam2} gives
\be \label{unifomr-008}
\ba
\|\mathbf{u}\|_{H^2} & \leq C \|\rho  \mathbf{u}_t \|_{L^2} + C \| (\rho \mathbf{u} \cdot \nabla ) \mathbf{u} \|_{L^2} + C \|\mathbf{u} \|_{H^1} \\
& \leq C \|\sqrt{\rho} \mathbf{u}_t \|_{L^2} + C \|\rho\|_{L^\infty} \|\mathbf{u}\|_{L^\infty} \|\nabla \mathbf{u}\|_{L^2} + C \|\mathbf{u}\|_{H^1} \\
& \leq C \|\sqrt{\rho } \mathbf{u}_t \|_{L^2} + C \| \mathbf{u}\|_{H^2}^{\frac12} \|\mathbf{u} \|_{H^1}^{\frac32} + C \|\mathbf{u}\|_{H^1}.
\ea
\ee
It follows from Young's inequality and \eqref{L2-H1} that 
\be \label{uniform-009} \ba
\ \ & \|\mathbf{u}\|_{H^2}  \leq C ( \|\sqrt{\rho } \mathbf{u} \|_{L^2} +  \|\mathbf{u}\|_{H^1}^3 + \|\mathbf{u}\|_{H^1}) \\
 \leq & C \left(\|\sqrt{\rho } \mathbf{u}_t \|_{L^2}+  \left[ \bar{\rho}^{-1} \int \rho |\mathbf{u}|^2 \, dx + \bar{\rho}^{-1} \|\rho - \bar{\rho} \|_{L^{\frac32}} \|\nabla \mathbf{u}\|_{L^2}^2 + \|\nabla \mathbf{u} \|_{L^2}^2 \right]^{\frac32} + 1\right) \\
 \leq & C (\|\sqrt{\rho } \mathbf{u}_t \|_{L^2}  + \|\nabla \mathbf{u}\|_{L^2 }^3+1 ).
\ea
\ee
Hence, substituting \eqref{uniform-009} and \eqref{L2-H1} into \eqref{uniform-007}     gives
\be \label{uniform-010}
 \frac{d}{dt} \int  |\nabla \mathbf{u} |^2 \, dx + \int  \rho |\mathbf{u}_t|^2 \, dx
 \leq C_4  (1 + \|\nabla \mathbf{u}\|_{L^2}^2)^3 .
\ee

Let $T_0^* = \frac{1}{8 C_4 \left(1 + \|\nabla \mathbf{u}_0 \|_{L^2}^2 \right)}$. It follows from \eqref{uniform-010} that
\be \label{uniform-011}
\sup_{0\leq t \leq \min\{T_0, T_0^* \}  } \left( \|\nabla \mathbf{u} \|_{L^2}^2 + 1 \right) + \int_0^{\min\{T_0, T_0^*\}}
\int \rho |\mathbf{u}_t|^2 \, dx dt \leq 2 \left( \|\nabla \mathbf{u}_0\|_{L^2}^2 + 1 \right) . \ee
According to the blow up criterion obtained in \cite{Kim}, the estimate \eqref{uniform-011} implies that the local strong solution does not blow up before the time $T_0^*$, i.e.,  $T_0 \geq T_0^*$.

Moreover, combining \eqref{uniform-011}  and  \eqref{uniform-009} together  gives that
\be \label{uniform-012}
\sup_{0<T < T_0^* } \int_0^T \|\mathbf{u}\|_{H^2}^2 \, dt \leq C .
\ee
Thus the proof of the lemma is completed.
\end{proof}

\begin{lemma}\label{uniform-4}
There exists a constant $C$, which depends on the $C^3$-regularity of $\partial \tilde{\Omega}$,  $\|\rho_0\|_{L^\infty}$, $\|\rho_0 - \bar{\rho}\|_{L^{\frac32}}$,
 $\|\rho_0 -\bar{\rho} \|_{H^2}$, $\| \mathbf{u}_0 \|_{H^2}$, $\bar{\rho}^{-1}$, $T_0^*$, such that for every $0< T< T_0^*$,
\be \nonumber
\sup_{0\leq  T \leq T_0^*} \left\{ \|\sqrt{\rho } \mathbf{u}_t \|_{L^2} + \|\mathbf{u}\|_{H^2} + \| \nabla \rho \|_{H^1}
+ \int_0^T \|\nabla \mathbf{u}_t \|_{L^2}^2 \, dt   \right\} \leq C .
\ee
\end{lemma}
The proof for Lemma \ref{uniform-4} follows exactly the same as that for Propositions \ref{step 4 lemma}--\ref{step 5 lemma}.

Now we are in position to prove the  local existence of strong solutions for the inhomogeneous Navier-Stokes system in an exterior domain.

\begin{theorem}
\label{local} Assume that $(\rho_0, \mathbf{u}_0)$ satisfies conditions (\ref{regular condition}) and (\ref{compatibility}),
then there exist a positive time $T_0^* $ and a unique strong solution $(\rho
,\mathbf{u})$ to the initial boundary value problem \eqref{NS}-\eqref{IBVP}
satisfying%
\begin{equation*}
\rho- \bar{\rho}  \in C([0,T_0^*);H^{2}(\Omega )),\text{ \ }\mathbf{u}\in
C([0,T_0^*; H_{0, \sigma}^{1}(\Omega )) \cap L^\infty(0, T_0^*);  H^2 (\Omega )),
\end{equation*}%
\begin{equation*}
\rho_t \in L^\infty(0, T_0^*; H^1(\Omega)), \ \ \sqrt{\rho }\mathbf{u}%
_{t}\in L^{\infty }(0,T_0^* ; L^{2}(\Omega ))\ \ \nabla \mathbf{u}_t \in L^2(0, T_0^*; L^2(\Omega)).
\end{equation*}%
\end{theorem}
\begin{proof}
Given $k\in \mathbb{N}$, let $ \Omega_k : =\Omega\cap \{|x|< k\}$. In each domain $\Omega_k$, choose the initial density and velocity $(\rho_{0, k}, \mathbf{u}_{0, k})$, which satisfy that
\be\nonumber
\rho_{0, k}= \rho_0 + \epsilon_{k}, \ \ \mbox{with}\ \lim_{k\rightarrow \infty} \| \epsilon_k \|_{H^2(\Omega_k)}= 0,\ \ \ \inf_{\Omega_k}\epsilon_k>0,
\ee
\be
\mathbf{u}_{0,k} \in H^2(\Omega_k)\cap H_{0, \sigma}^1(\Omega_k),\ \ \  \mbox{with}\
\| \mathbf{u}_{0,k}\|_{H^2(\Omega_k)} \leq 2 \|\mathbf{u}_0 \|_{H^2(\Omega)},
\ee
and
\be \nonumber
\mathbf{u}_{0,k} \ \mbox{converges to } \mathbf{u}_0\  \mbox{in}\ H^2(\Omega^\prime), \  \ \mbox{ for each compact subdomain} \ \Omega^{\prime}.
\ee
By Lemma \ref{localexistence}, for each $k\in \mathbb{N}$,  there exists a unique strong solution $(\rho_k, \mathbf{u}_k)$  to the equations \eqref{NS} with the initial data $(\rho_{0, q}, \mathbf{u}_{0,k})$ over some time interval $[0, T_k)$. As proved above, there exists a positive time  $T_0^*$, which depends only on the $C^3$-regularity of $\partial \Omega_k$, $\|\rho_{0, k}- \bar{\rho } \|_{L^{\frac32}(\Omega_k)}$, $\bar{\rho}^{-1}$, $ \|\rho_{0,k} \|_{L^\infty(\Omega_k)}$, $\|\mathbf{u}_{0,k}\|_{H^1(\Omega_k)} $, such that
\be \nonumber
T_k \geq T_0^*.
\ee
It means that the lifespans of the approximate solutions $(\rho_k, \mathbf{u}_k)$ have a uniform lower bound $T_0^*$. Moreover, as proved above,
there exists some constant $C$ which does not depend on the size of $\Omega$ or $\partial \Omega$, such that
\be \label{uniform-021}
\sup_{ 0 \leq T \leq T_0^* } \left( \|\rho_k - \bar{\rho} \|_{L^\frac32(\Omega_k)} + \|\rho_k - \bar{\rho}\|_{H^2(\Omega_k)}+ \|\partial_t \rho_k\|_{H^1(\Omega_k)} + \|\mathbf{u}_k \|_{H^2(\Omega_k)} \right)
\leq C,
\ee
and
\be \label{unifrom-22}
\int_0^{T_0^*} \|\nabla \partial_t \mathbf{u}_{k} \|_{L^2(\Omega_k)}^2 \, dt \leq C .
\ee
 Hence, there exists a subsequence of $(\rho_k, \mathbf{u}_k)$(which is still labelled by $(\rho_k, \mathbf{u}_k)$), and the limit function $(\rho, \mathbf{u})$, such that for every compact subdomain $\Omega^{\prime}$,
\be \nonumber
\begin{array}{c}
\rho_k - \bar{\rho } \stackrel{\ast}{\rightharpoonup}\ \rho - \bar{\rho} \ \ \mbox{in}\ L^\infty(0, T_0^*; H^2 (\Omega^{\prime} )), \ \ \ \ \mathbf{u}_k \stackrel{\ast}{\rightharpoonup} \ \mathbf{u}\ \  \mbox{in}\ L^\infty(0, T_0^*; H^2(\Omega^{\prime} )), \\[2mm]
\partial_t \rho_k  \stackrel{\ast}{\rightharpoonup}\  \partial_t \rho \ \ \mbox{in}\ L^\infty(0, T_0^*; H^1(\Omega^{\prime} )), \ \ \ \ \nabla \partial_t \mathbf{u}_k
\rightharpoonup \ \nabla \partial_t \mathbf{u}\ \ \mbox{in}\ L^2(0, T_0^*; L^2(\Omega^{\prime})).
\end{array}
\ee
and
\be \nonumber
\sup_{0\leq T \leq T_0^*} \left(  \|\rho - \bar{\rho}\|_{L^{\frac32}(\Omega)} + \|\rho - \bar{\rho} \|_{H^2(\Omega)} + \|\mathbf{u}\|_{H^2(\Omega)}
+ \|\rho_t\|_{H^1(\Omega)} \right) \leq C,
\ee
\be \nonumber
\int_0^{T_0^*} \|\nabla \mathbf{u}_t \|_{L^2(\Omega)}^2 \, dt \leq C.
\ee
By Aubin-Lions Lemma,
\be \nonumber
\rho_k - \bar{\rho} \rightarrow \ \rho - \bar{\rho} \ \mbox{in}\ C ([0, T_0^*]; H^1(\Omega^{\prime})),\ \ \ \ \mathbf{u}_k \rightarrow \mathbf{u}\ \mbox{in}\ C([0, T_0^*]; H^1(\Omega^{\prime})).
\ee
Hence, $(\rho, \mathbf{u})$ is a strong solution to the initial value problem \eqref{NS}--\eqref{IBVP}. Furthermore, it follows from the equations that $\rho - \bar{\rho} \in C ([0, T_0^*]; H^2(\Omega))$. The proof of uniqueness is now standard and one can refer to \cite{Choe} for details. Hence, the proof of Theorem \ref{local} is completed.
\end{proof}

\medskip

{\bf Acknowledgement.}
The research of Guo was partially supported by China Postdoctoral Science Foundation grants 2017M620149 and 2018T110387.
The research of Wang was partially supported by NSFC grant 11671289. The research of  Xie was partially supported by  NSFC grants 11971307 and 11631008,  and Young Changjiang Scholar of Ministry of Education in China.

\end{document}